\documentclass[11pt,a4paper]{amsart}

\usepackage[utf8]{inputenc}

\usepackage{hyperref,graphicx,amssymb,enumerate}

\newcommand*{\LargerCdot}{\raisebox{-0.25ex}{\scalebox{1.8}{$\cdot$}}}
\newcommand{\abs}[1]{\lvert#1\rvert}

\DeclareMathOperator{\Sym}{Sym}
\DeclareMathOperator{\Sp}{Sp}
\DeclareMathOperator{\Or}{O}

\newtheorem{theorem}{Theorem}
\newtheorem{lemma}[theorem]{Lemma}
\newtheorem{corollary}[theorem]{Corollary}
\newtheorem*{theorem*}{Theorem}

\theoremstyle{remark}
\newtheorem{remark}[theorem]{Remark}

\begin{document}

\title[The non-existence of sharply $2$-transitive sets in $\mathrm{Sp}(2d,2)$]{The non-existence of sharply $2$-transitive sets of permutations in $\mathrm{Sp}(2d,2)$ of degree $2^{2d-1}\pm 2^{d-1}$}

\author{Dominik Barth}

\address{Institute of Mathematics\\ University of Würzburg \\ Emil-Fischer-Straße 30 \\ 97074 Würzburg, Germany}

\email{dominik.barth@mathematik.uni-wuerzburg.de}



\begin{abstract}
We use Müller and Nagy's method of contradicting subsets to give a new proof for the non-existence of sharply $2$-transitive subsets of the symplectic groups $\mathrm{Sp}(2d,2)$ in their doubly-transitive actions of degrees $2^{2d-1}\pm 2^{d-1}$. The original proof by Grundhöfer and Müller was rather complicated and used some results from modular representation theory, whereas our new proof requires only simple counting arguments.
\end{abstract}

\maketitle

\section{Sharply \texorpdfstring{$2$}{2}-transitive sets in finite permutation groups}

Let $\Omega$ be a finite set. A set $S \subseteq \Sym \Omega$ of permutations is called \emph{sharply transitive} if for each $\alpha, \beta \in \Omega$ there is a unique $g\in S$ with $\alpha^g=\beta$. We call $S$ \emph{sharply $2$-transitive} on $\Omega$ if it is sharply transitive on the set of pairs $(\omega_1,\omega_2)\in \Omega^2$ with $\omega_1 \neq \omega_2$.

Sharply transitive and sharply $2$-transitive sets of permutations correspond to Latin squares and projective planes, respectively, and are therefore relevant to the fields of combinatorics and geometry.

In the 1970s, Lorimer started the classification of all finite permutation groups containing a sharply $2$-transitive subset, see \cite{Lorimer1974} for a summary of his work.
In 1984, O'Nan continued Lorimer's research and mentions in \cite{ONan1984} that ``work on the groups $\Sp(2d,2)$ of degrees $2^{2d-1}\pm 2^{d-1}$ is in progress'' --- these are the permutation groups our paper will mainly focus on. 
Later in 2009, Grundhöfer and Müller \cite{Grundhoefer2009} extended O'Nan's character-theoretic methods and finally showed --- among many non-existence results for other almost simple finite permutation groups --- that there are no sharply $2$-transitive sets in $\Sp(2d,2)$ of degree $2^{2d-1}\pm 2^{d-1}$. Their complicated proof was based on results by Sastry and Sin \cite{Sastry2002} regarding characteristic $2$ permutation representations for symplectic groups.

In 2010, Müller and Nagy \cite{Mueller2011} introduced the so-called method of \emph{contradicting subsets} and were --- in addition to proving new results --- also able to give simple combinatorial proofs for most of the previously known results by Lorimer, O'Nan, Grundhöfer and Müller.

However, they didn't reprove the non-existence of sharply $2$-transitive sets in the symplectic groups $\Sp(2d,2)$ of degrees $2^{2d-1} \pm 2^{d-1}$. This is done in the present paper.
More precisely, we construct contradicting subsets for the orthogonal groups $\Or^\pm(2d,2)$, $d \geq 4$, in their natural actions on singular vectors --- which happen to be the point stabilizers of the symplectic groups.
	
The paper concludes with a remark on limitations of the contradicting subset method and we conjecture that contradicting subsets cannot be used to prove the non-existence of sharply $2$-transitive sets in the full symmetric groups $S_n$.

\section{Contradicting subsets}

In this short section we recap the basic facts about the simple yet powerful method of contradicting subsets. We closely follow \cite{Mueller2011}.

Let $G\leq \Sym \Omega$ be a group of permutations on a finite set $\Omega$ with $\abs{\Omega}=n$.
We denote the permutation matrix of $g \in G$ by $P_g \in \lbrace 0,1 \rbrace^{n \times n}$ and the $n \times n$ all-one matrix by $J$.
Using these definitions, sharply transitive subsets $S \subseteq G$ clearly correspond to $0$-$1$-solutions of the following system of linear equations:
  
\begin{equation}\label{eq:J} \tag{J}
\sum_{g \in G} x_g P_g = J.
\end{equation}
The next simple lemma is the main ingredient of Müller and Nagy's contradicting subset method.

\begin{lemma}\label{lem:contr}
Let $S \subseteq \Sym \Omega$ be sharply transitive and let $B$, $C$ be arbitrary subsets of $\Omega$. Then $\sum_{g \in S} \abs{B \cap C^g} = \abs{B}\abs{C}$.
\end{lemma}

We often use Lemma~\ref{lem:contr} in the following way:

\begin{corollary}
Let $G$ be a permutation group on a finite set $\Omega$. Assume that there are subsets $B,C$ of $\Omega$ and a natural number $k$ such that $k$ divides $\abs{B \cap C^g}$ for all $g \in G$ but does not divide $\abs{B}\abs{C}$. Then $G$ does not contain a sharply transitive subset. In this situation we call $B,C$ \emph{contradicting subsets} for $G$ (modulo $k$).
\end{corollary}

\begin{remark}\label{rem:modk}
It is easy to see that the existence of contradicting subsets modulo $k$ implies the unsolvability of \eqref{eq:J} in the ring $\mathbb{Z}/k\mathbb{Z}$ and therefore in particular over the integers.
\end{remark}

\section{The actions of \texorpdfstring{$\Sp(2d,2)$}{Sp(2d,2)} of degrees \texorpdfstring{$2^{2d-1} \pm 2^{d-1}$}{2\textasciicircum(2d-1)\textpm2\textasciicircum(d-1)}}

In this section we describe the doubly transitive actions of the symplectic groups $\Sp(2d,2)$ of degrees $2^{2d-1} \pm 2^{d-1}$, sometimes called \emph{Jordan-Steiner} actions.
We roughly follow the standard arguments given in \cite[chapter 2]{Sastry2002}.
For other references regarding this topic, see \cite[chapter 7.7]{Dixon1996} and \cite{Mortimer1980}.

Fix $d \geq 2$ and let $V$ be a $2d$-dimensional vector space over $\mathbb{F}_2$ equipped with a symplectic (i.e.\ alternating, bilinear and non-degenerate) form $\varphi$. Let further $\Omega$ denote the set of quadratic forms $\theta : V \rightarrow \mathbb{F}_2$ which polarize to $\varphi$, that is $\Omega = \lbrace \theta:V \rightarrow \mathbb{F}_2 \mid \theta(u+v)=\theta(u)+\theta(v)+\varphi(u,v)~ \forall u,v \in V \rbrace$. The symplectic group $\Sp(V,\varphi)$ acts on $\Omega$ via $\theta^g(v)=\theta(v^{g^{-1}})$.

Now fix some $\theta \in \Omega$. Then for all $\theta' \in \Omega$ the map $\theta-\theta':V \rightarrow \mathbb{F}_2$ is linear and can thus be written as $\theta-\theta'=\varphi(\LargerCdot,a)$ for some $a \in V$. It follows that $\Omega$ consists of the elements $\theta_a$, $a \in V$, where $\theta_a(v):=\theta(v)+\varphi(v,a)$.

The point stabilizer $\Sp(V,\varphi)_\theta$ of $\theta \in \Omega$ in this action is the orthogonal group $\Or(V,\theta)$. Furthermore, the action of $\Or(V,\theta)$ on $\Omega$ is equivalent to the natural action of $\Or(V,\theta)$ on $V$ if we identify $v \in V$ with $\theta_v \in \Omega$. (The equivalence follows from $(\theta_v)^g(x)=\theta_v(x^{g^{-1}})=\theta(x^{g^{-1}})+\varphi(x^{g^{-1}},v)=\theta(x)+\varphi(x,v^g)=\theta_{v^g}(x)$ for $x,v \in V$ and $g \in \Or(V,\theta)$)

The quadratic forms in $\Omega$ fall into two categories corresponding to their Witt-Index. Let $\Omega_+$ denote the forms of Witt-Index $d$ and $\Omega_-$ those of Witt-Index $d-1$. It is clear that both $\Omega_+$ and $\Omega_-$ are invariant under $\Sp(V,\varphi)$. Choose $\varepsilon \in \lbrace -,+\rbrace$ such that $\theta \in \Omega_\varepsilon$. Then  \cite[equation 5]{Sastry2002} or \cite[Lemma 1]{Dye1979} yield $\Omega_\varepsilon = \lbrace \theta_v \mid \theta(v)=0 \rbrace$, in other words $\theta$ and $\theta_v$ have the same Witt-Index if and only if $\theta(v)=0$.

According to Witt's Lemma, the orthogonal group $\Or(V,\theta)$ has three orbits on $V$, namely $\lbrace 0 \rbrace$, $\lbrace v \in V \mid \theta(v)=0, v \neq 0 \rbrace$ and $\lbrace v \in V \mid \theta(v)=1 \rbrace$. Via $V \ni v \mapsto \theta_v \in \Omega$, we conclude that $\Or(V, \theta)$ has the following three orbits on $\Omega$: $\lbrace \theta \rbrace$, $\Omega_\varepsilon \setminus \lbrace \theta \rbrace$ and $\Omega_{-\varepsilon}$. Thus the symplectic group $\Sp(V,\varphi)$ is doubly transitive on both $\Omega_+$ and $\Omega_-$. Furthermore, it is well-known that $\abs{\Omega_\pm} = 2^{2d-1} \pm 2^{d-1}$.

In the next section we prove the non-existence of sharply $2$-transitive sets in $\Sp(2d,2)$ by showing the non-existence of sharply transitive sets in the point stabilizers of $\Sp(2d,2)$ on the remaining points. The foregoing discussion shows that these are the orthogonal groups $\Or^\pm(2d,2)$ in their natural actions on the singular vectors.

\section{Contradicting subsets for orthogonal groups}\label{sec:contr_orth}

In order to prove that the orthogonal group $\Or^\pm(2d,2)$, $d \geq 4$, has no sharply transitive subset on the set of singular vectors, we construct a pair $B,C$ of contradicting subsets.

Fix $d \geq 4$ and let $V$ be a $2d$-dimensional vector space equipped with an quadratic form $\theta$ that polarizes to a symplectic form $\varphi$.
For a subset $X \subseteq V$ we define $X^0 := \lbrace v \in X \mid \theta(v)=0, v \neq 0 \rbrace$ and $X^1 := \lbrace v \in X \mid \theta(v)=1 \rbrace$.
We want to construct subsets $B, C \subseteq V^0$ such that $2^{d-1}$ divides $\abs{B \cap C^g}$ for all $g \in \Or(V,\theta)$ but $2^{d-1}$ does not divide $\abs{B}\abs{C}$.

To do this, we decompose $V = U \bot W$ into non-degenerate subspaces $U$, $W$ of even dimensions $\dim U, \dim W \geq 4$. This is always possible, as we can choose $U$ as an orthogonal sum of two or more hyperbolic planes and set $W := U^\bot$.

Now we are able to define our contradicting subsets for $\Or^\pm(2d,2)$ and state the main theorem.

\begin{theorem}\label{th:main}
Let $B:= U^1 + W^1 = \lbrace u+w \mid u\in U, w\in W, \theta(u)=\theta(w)=1 \rbrace$ and $C:= V^0 \cap c^\bot = \lbrace v \in V \mid \theta(v)=0, \varphi(v,c)=0, v \neq 0 \rbrace$ for some fixed $c \in V^1$.
Then the following holds.
\begin{enumerate}[(i)]
\item $2^{d-1}$ divides $\abs{B \cap C^g}$ for all $g \in \Or(V,\theta)$ but
\item $2^{d-1}$ is no divisor of $\abs{B}\abs{C}$.
\end{enumerate}
Thus $B,C$ are contradicting subsets for the action of $\Or(V,\theta)$ on the singular vectors $V^0$.
Therefore, the orthogonal groups $\Or^\pm(2d,2)$, $d \geq 4$, contain no sets that are sharply transitive on the singular vectors.
\end{theorem}

As a stabilizer $\Sp(2d,2)_\theta$ is an orthogonal group we immediately obtain.

\begin{corollary}
The symplectic group $\Sp(2d,2)$, $d \geq 4$, in one of its two doubly transitive actions of degree $2^{2d-1}\pm 2^{d-1}$, has no sharply 2-transitive subset.
\end{corollary}

\begin{remark}
The first contradicting subsets for $\Sp(8,2)$ have been found with a Magma \cite{Bosma1997} program that was searching the orbits of maximal subgroups of $\Or^-(8,2)$ for contradicting subsets. The construction given above is a generalization of those contradicting subsets to $\Or^\pm(2d,2)$ for all $d \geq 4$.
\end{remark}

\section{Proof of the main theorem}

First we introduce a notation: If $(V,\theta)$ is a $2d$-dimensional non-degenerate orthogonal space over $\mathbb{F}_2$, we set $\pm_V := +$ if $\theta$ has Witt-Index $d$ and $\pm_V:= -$ in case of Witt-Index $d-1$. Let further $\mp_V$ denote the opposing sign $-\pm_V$.

In order to prove Theorem~\ref{th:main}, we need to do some simple combinatorics in orthogonal spaces. First, we calculate the number of nonsingular vectors perpendicular to a given vector.

\begin{lemma}\label{lem:comb}
Let $V$ be a $2d$-dimensional non-degenerate orthogonal space over $\mathbb{F}_2$.
Then the following holds.
\begin{enumerate}[(i)]
\item $\abs{V^0}=2^{2d-1}\pm_V 2^{d-1}-1$ and $\abs{V^1}=2^{2d-1}\mp_V 2^{d-1}$.
\item For $v \in V^1$ we have $\abs{V^1 \cap v^\bot} = \abs{(V^0 \cap v^\bot) \dot{\cup} \lbrace 0 \rbrace}= 2^{2d-2}$.
\item For $v \in V^1$ we have $\abs{V^1\setminus v^\bot}=2^{2d-2} \mp_V 2^{d-1}$.
\item For $v \in V^0$ we have $\abs{V^1 \setminus v^\bot}=2^{2d -2}$.
\item For $v \in V^0$ we have $\abs{V^1 \cap v^\bot}= 2^{2d-2} \mp_V 2^{d-1}$.
\end{enumerate}
\end{lemma}
\begin{proof} ~
\begin{enumerate}[(i)]
\item This result is well-known, see for example \cite[Proposition 14.47]{Grove2001} or \cite[Theorem 11.5]{Taylor1992}.

\item The mapping
$x \mapsto x + v$
is a bijection between $V^1 \cap v^\bot$ and $(V^0 \cap v^\bot) \cup \lbrace 0 \rbrace$.
Thus we obtain $\abs{(V^0 \cap v^\bot) \cup \lbrace 0 \rbrace} =\abs{V^1 \cap v^\bot} = \frac{1}{2}\abs{v^\bot} = 2^{2d -2}$.

\item
We use (i) and (ii):
\begin{align*}
\abs{V^1 \setminus v^\bot} = \abs{V^1} - \abs{V^1 \cap v^\bot} = 2^{2d-1}\mp_V 2^{d-1} - 2^{2d -2} = 2^{2d-2} \mp_V 2^{d-1}.
\end{align*}

\item
In this case the mapping
$x \mapsto x + v$
is a bijection between $V^1 \setminus v^\bot$ and $V^0 \setminus v^\bot$.
Thus $\abs{V^1 \setminus v^\bot} = \frac{1}{2} \abs{V \setminus v^\bot} = \frac{1}{2} (\abs{V}-\abs{v^\bot}) = 2^{2d -2}$.

\item
follows from (i) and (iv):
\begin{align*}
\abs{V^1 \cap v^\bot} = \abs{V^1} - \abs{V^1 \setminus v^\bot} = 2^{2d-1}\mp_V 2^{d-1} - 2^{2d -2}=2^{2d-2} \mp_V 2^{d-1}.
\end{align*}
\end{enumerate}
\end{proof}

Using Lemma~\ref{lem:comb} we conclude the following result which is important for the proof of the main theorem.

\begin{lemma}\label{lem:divides}
Let $(U,\theta)$ and $(W,\theta')$ be non-degenerate orthogonal spaces over $\mathbb{F}_2$ with dimensions $\dim U=2a$ and $\dim W=2b$. Further choose $u \in U$ and $w \in W$ with $\theta(u)+\theta'(w)=1$. If $a,b \geq 2$ then $2^{a+b-1}$ divides 
\begin{align*}
\abs{U^1 \cap u^\bot}\cdot \abs{W^1 \cap w^\bot}  + \abs{U^1 \setminus u^\bot} \cdot \abs{W^1 \setminus w^\bot}.
\end{align*}
\end{lemma}
\begin{proof}
The symmetry of $U$ and $W$ allows us to choose $u \in U^1$ without loss of generality. Now two cases remain for $w$: $w\in W^0$ or $w=0 \in W$:

If $w\in W^0$, then we use Lemma~\ref{lem:comb} to calculate
\begin{align*}
\abs{U^1 \cap u^\bot} \cdot \abs{W^1 \cap w^\bot}
= 2^{2a-2} \cdot (2^{2b-2} \mp_W 2^{b-1})
= 2^{2a+2b -4} \mp_W 2^{2a+b-3}
\end{align*}
and
\begin{align*}
\abs{U^1 \setminus u^\bot} \cdot \abs{W^1 \setminus w^\bot}
= (2^{2a-2} \mp_U 2^{a-1})\cdot 2^{2b-2}
= 2^{2a+2b-4} \mp_U 2^{2b+ a-3}.
\end{align*}
Both terms are divisible by $2^{a+b-1}$. The claim follows.

In the case of $w=0 \in W$, we have $w^\bot = W$ and obtain from Lemma~\ref{lem:comb}
\begin{align*}
\abs{U^1 \cap u^\bot} \cdot \abs{W^1 \cap w^\bot} &+ \abs{U^1 \setminus u^\bot} \cdot \abs{W^1 \setminus w^\bot} = \abs{U^1 \cap u^\bot} \cdot \abs{W^1} \\
&= 2^{2a-2} \cdot (2^{2b-1} \mp_W 2^{b-1})
= 2^{2a+2b-3} \mp_W 2^{2a+b-3},
\end{align*}
from which the divisibility follows immediately.
\end{proof}

Now recall the situation in section~\ref{sec:contr_orth}: The $2d$-dimensional orthogonal space $V$ is decomposed into $V=U \bot W$ and we defined $B = U^1+W^1$, $C=V^0 \cap c^\bot$ for $c \in V^1$. Writing $\dim U=2a$ and $\dim W =2b$ we have $a,b \geq 2$ and $a+b=d$.

\begin{lemma}\label{lem:B}
$B$ contains only singular vectors, that is $B \subseteq V^0$, and $2^{d-1}$ does not divide $\abs{B}$.
\end{lemma}
\begin{proof}
$B \subseteq V^0$ is obvious.
From $V = U \oplus W$ and Lemma~\ref{lem:comb} we conclude
\begin{align*}
\abs{B} & = \abs{U^1} \cdot \abs{W^1} = (2^{2a-1} \mp_U  2^{a-1}) \cdot (2^{2b-1} \mp_W  2^{b-1}) \\
&= 2^{2a+2b -2} \mp_W 2^{2a+b-2}\mp_U 2^{a+2b-2} \mp_U \mp_W 2^{a+b-2} \\
&= 2^{2d -2} \mp _W 2^{d+a-2} \mp_U 2^{d+b-2} \mp_U \mp_W 2^{d-2}.
\end{align*}

Now the claim follows because $2^{d-1}$ divides $2^{2d -2}$, $2^{d+a-2}$ and $2^{d+b-2}$, but does certainly not divide $2^{d-2}$.
\end{proof}

We need the following lemma in order to prove the divisibility of $\abs{B \cap C^g}$ by $2^{d-1}$ for all $g \in \Or(V, \theta)$.
\begin{lemma}\label{lem:dividesBcapC}
$2^{d-1}$ divides $\abs{B\cap z^\bot}$ for all $z\in V^1$.
\end{lemma}
\begin{proof}
We write $z$ as $z=x+y$ with $x \in U$ and $y\in W$ and calculate $\abs{B \cap z^\bot}$ using the decomposition $V=U \bot W$:
\begin{align*}
\abs{B\cap z^\bot} &= \abs{\lbrace u+w \mid u \in U^1,w\in  W^1 , \varphi(u+w,z)=0 \rbrace} \\
&=\abs{\lbrace (u,w) \in U^1 \times W^1 \mid \varphi(u,z)+\varphi(w,z)=0 \rbrace}\\
&=\abs{\lbrace (u,w) \in U^1 \times W^1 \mid \varphi(u,x+y)+\varphi(w,x+y)=0 \rbrace}\\
&=\abs{\lbrace (u,w) \in U^1 \times W^1 \mid \varphi(u,x)+\varphi(w,y)=0 \rbrace} \\
&= \abs{U^1 \cap x^\bot} \cdot \abs{W^1 \cap y^\bot} + \abs{U^1 \setminus x^\bot} \cdot \abs{W^1 \setminus y^\bot}.
\end{align*}
Further holds
\begin{align*}
1=\theta(z)=\theta(x+y)=\theta(x)+\theta(y)+\varphi(x,y)=\theta(x)+\theta(y)
\end{align*}
and using Lemma~\ref{lem:divides} we obtain the desired
\begin{align*}
2^{d-1} = 2^{a+b-1} \mid \abs{U^1 \cap x^\bot} \cdot \abs{W^1 \cap y^\bot} + \abs{U^1 \setminus x^\bot} \cdot \abs{W^1 \setminus y^\bot} = \abs{B \cap z^\bot}.
\end{align*}
\end{proof}

Now we are able to prove the main theorem.

\begin{theorem*}
The subsets $B,C \subseteq V^0$ have the following properties.
\begin{enumerate}[(i)]
\item $2^{d-1}$ divides $\abs{B \cap C^g}$ for all $g \in \Or(V,\theta)$ but
\item $2^{d-1}$ is no divisor of $\abs{B}\abs{C}$.
\end{enumerate}
\end{theorem*}
\begin{proof}
For every $g\in \Or(V,\theta)$ we have
\begin{align*}
B \cap C^g = B \cap (V^0 \cap c^\bot)^g = B \cap V^0 \cap (c^\bot)^g = B \cap V^0 \cap (c^g)^\bot = B \cap (c^g)^\bot.
\end{align*}
Using Lemma~\ref{lem:dividesBcapC} with $c^g \in V^1$ yields that $\abs{B \cap (c^g)^\bot} = \abs{B \cap C^g}$ is divisible by $2^{d-1}$.
Furthermore, Lemma~\ref{lem:comb} yields
\begin{align*}
\abs{C} = \abs{V^0 \cap c^\bot}
= \abs{(V^0 \cap c^\bot) \cup \lbrace 0 \rbrace} -1
= 2^{2d -2} -1,
\end{align*}
and in particular $\abs{C}$ is odd. Thus $2^{d-1} \nmid \abs{B}\abs{C}$ follows from $2^{d-1} \nmid \abs{B}$ which was proven in Lemma~\ref{lem:B}.

\end{proof}

\section{Remark: Limitations of the contradicting subset method}

In \cite{Mueller2011}, Müller and Nagy gave many successful applications of their method of contradicting subsets, for example by showing the non-existence of sharply $2$-transitive sets in the Mathieu group $M_{23}$, the alternating groups $A_n$, $n \equiv 2,3 \bmod 4$, the Conway group $Co_3$ or in automorphism groups of nontrivial symmetric designs.

However, there remain some finite doubly transitive permutation groups where the question of existence of sharply $2$-transitive subsets is still open. The only almost simple groups among them are the largest Mathieu group $M_{24}$, the alternating groups $A_n$ for $n \equiv 0,1 \bmod 4$ and the symmetric groups $S_n$ (of course apart from prime powers, $n=10$, and the exceptions through the Bruck-Ryser-Theorem).

It is therefore interesting to ask, whether it is possible to find contradicting subsets for these permutation groups. Unfortunately, it is often the case that the corresponding system $\eqref{eq:J}$ of linear equations is solvable over the integers which implies the non-existence of contradicting subsets, see Remark~\ref{rem:modk}.

In \cite[4. Remarks on $M_{24}$]{Mueller2011}, it is shown that $\eqref{eq:J}$ has an integer solution for $M_{24}$ of degree $24 \cdot 23$. The remaining part of our paper discusses the integral solvability of $\eqref{eq:J}$ for the full symmetric group $S_n$ of degree $n(n-1)$. More precisely, we give a system of linear equations consisting of roughly $n$ unknowns and $3$ equations whose integer solvability implies that of $\eqref{eq:J}$ for $S_n$ of degree $n(n-1)$. Using this, computer calculations show the integer solvability of $\eqref{eq:J}$ for $S_n$ of degree $n(n-1)$ for at least all $n \leq 100$.

In our linear equations, certain coefficients $\Gamma_{i,j}^{(n)}$ play an important role.
With the convention $\gcd(\emptyset)=0$, these are defined as 
\begin{align*}
\Gamma_{i,j}^{(n)} := \gcd \left( \frac{n!}{\prod_{k=1}^n k^{e_k} e_k!} : e_1=i,e_2=j,\sum_{k=1}^n ke_k = n, e_k \in \lbrace 0,\dots, n \rbrace \right).
\end{align*}
Note that the coefficient $\Gamma_{i,j}^{(n)}$ is just the greatest common divisor of the sizes of all conjugacy classes in $S_n$ having exactly $i$ $1$-cycles and $j$ $2$-cycles.
Now set $m=\lfloor n/2 \rfloor$ and consider the following system of linear equations with unknowns $y_{i,j}$, $i=0,1$, $j=0, \dots, m$:

\begin{align}\label{eq:zsolvesym}
\begin{split}
\sum_{j = 0}^m y_{0,j} j\Gamma_{0,j}^{(n)} + \sum_{j = 0}^m y_{1,j} j\Gamma_{1,j}^{(n)} &= n(n-1)/2\\
\sum_{j = 0}^m y_{0,j}\Gamma_{0,j}^{(n)} &= n-1\\
\sum_{j = 0}^m y_{1,j}\Gamma_{1,j}^{(n)} &=  n(n-2)
\end{split}
\end{align}
We note here without proof that \eqref{eq:zsolvesym} corresponds to the system \eqref{eq:J} for $S_n$, $n \geq 4$, of degree $n(n-1)$ where we only consider integer solutions $x_g$, $g \in S_n$, that are constant on conjugacy classes and are zero on non-identity classes having more than one fixed point on $\lbrace 1, \dots, n \rbrace$.
In particular, the $\mathbb{Z}$-solvability of \eqref{eq:zsolvesym} implies that of \eqref{eq:J} for $S_n$ of degree $n(n-1)$.

To give a weaker condition for the solvability of \eqref{eq:J} over the integers, consider the $\mathbb{Z}$-module

\begin{align*}
M_n := \sum_{j=0}^{\lfloor n/2 \rfloor} \Gamma_{0,j}^{(n)}
\begin{pmatrix}
1 \\ j
\end{pmatrix}
\mathbb{Z}
\end{align*}
and the vectors
\begin{align*}
v_n &:= 
\begin{pmatrix}
n-1 \\ (n^3-6n^2+9n-3)(n-3)n(n-1)/2
\end{pmatrix}, \\
w_n &:=
\begin{pmatrix}
n-1 \\ -(n^2-n-1)(n-3)n(n-1)/2
\end{pmatrix}.
\end{align*}
It is not difficult to see that $v_n \in M_n$ and $w_{n-1}\in M_{n-1}$ imply the $\mathbb{Z}$-solvability of \eqref{eq:zsolvesym} and therefore of \eqref{eq:J} for $S_n$ of degree $n(n-1)$. 
Computer calculations show $v_n \in M_n$ and $w_{n-1}\in M_{n-1}$ for all $n \leq 100$ and it follows that contradicting subsets cannot be used to disprove those numbers to be orders of projective planes. We conjecture that this holds for all $n$ and pose it as an open problem.

\section*{Acknowledgement}

The author would like to thank Peter Müller for introducing him to the subject as well as for carefully reading earlier versions of this paper.

\end{document}